\documentclass[12pt]{article}
\usepackage{amssymb}
\usepackage{graphicx}
\usepackage{amsmath}

\newtheorem{theorem}{Theorem}

\newtheorem{lemma}[theorem]{Lemma}

\newtheorem{remark}[theorem]{Remark}
\newtheorem{corollary}[theorem]{Corollary}
\RequirePackage[latin1]{inputenc} \RequirePackage[T1]{fontenc}
\newenvironment{proof}[1][Proof]{\textbf{#1.} }{\ \rule{0.5em}{0.5em}}

\def \R{\mbox{I\hspace{-.15em}R}}
\def \P{\mbox{I\hspace{-.15em}P}}
\def \E{\mbox{I\hspace{-.15em}E}}
\def \D{\mbox{I\hspace{-.15em}D}}

\begin{document}

\author{Auguste Aman \thanks{augusteaman5@yahoo.fr}\\
UFR de Math\'{e}matiques et Informatique,\\
22 BP 582 Abidjan 22,\ C\^{o}te d'Ivoire\\
}\date{}
\title{Representation theorems for backward\\doubly stochastic differential equations}
\maketitle

\begin{abstract}
In this paper we study the class of backward doubly stochastic
differential equations (BDSDEs, for short) whose terminal value
depends on the history of forward diffusion. We first establish a
probabilistic representation for the spatial gradient of the
stochastic viscosity solution to a quasilinear parabolic SPDE in the
spirit of the Feynman-Kac formula, without using the derivatives of
the coefficients of the corresponding BDSDE. Then such a
representation leads to a closed-form representation of the
martingale integrand of BDSDE, under only standard Lipschitz
condition on the coefficients.
\end{abstract}

\textbf{Key Words}: Adapted solution, anticipating stochastic
calculus, backward doubly SDEs, stochastic partial
differential equation, stochastic viscosity solutions.

{\bf MSC}: 60H15; 60H20

\section{Introduction}

Backward stochastic differential equations (BSDEs, for short) were
firstly been considered in it linear form by Bismut \cite{B1,B2}
in the context of optimal stochastic control. However,  nonlinear
BSDEs and their theory have been introduced by Pardoux
and Peng \cite{PP1}. It has been enjoying a great interest in the last ten year because of
its connection with  applied fields. We can cite stochastic control and stochastic games (see
\cite{HL}) and mathematical finance (see \cite{kal}). BSDEs also provide a
probabilistic interpretation for solutions to elliptic or parabolic
nonlinear partial differential equations generalizing the classical
Feynman-Kac formula \cite{PP2,PP3}. A new class of BSDEs, called backward doubly stochastic differential equations (BDSDEs, in short), was considered by  Pardoux and Peng \cite{PP4}. This new kind of BSDEs present two stochastic integrals driven by two independent Brownian motions $B$ and $W$ and is of the form
\begin{eqnarray}
Y_{s}&=&\xi+\int_{s}^{T}f(r,Y_{r}, Z_{r})\;dr+\int_{s}^{T}g(r,Y_{r},
Z_{r})\downarrow dB_{r}
\notag \\
&&-\int_{s}^{T}Z_{r}\;dW_{r},\ s\in[t,T],  \label{a1}
\end{eqnarray}
where $\xi$ is a square integrable variable. Let us remark that in the sens of Pardoux Peng, the integral driven by $\{B_{r}\}_{r\geq 0}$ is a backward Itô integral and
the other one driven by $\{W_{r}\}_{r\geq 0}$ is the standard forward Itô integral. Further, backward doubly SDEs seem to be suitable giving a probabilistic representation for a system of parabolic stochastic partial differential equations (SPDEs, in short). We refer to Pardoux and Peng \cite{PP4} for the link between SPDEs and
BDSDEs in the particular case where solutions of SPDEs are regular. The general situation is much delicate to treat because of difficulties of extending the notion of viscosity solutions to SPDEs. The stochastic viscosity solution for semi-linear SPDEs was introduced for the first time
in Lions and Souganidis \cite{LS}. They used the so-called "stochastic characteristic" to remove the stochastic integrals from an SPDE. Another way of defining a stochastic viscosity solution of SPDE is via an appeal to the Doss-Sussman transformation. Buckdahn and Ma \cite{BM1,BM2} were the first to use this approach in order to connect the stochastic viscosity solutions of SPDEs with BDSDEs.

In this paper we consider the approach of defining stochastic viscosity solution of SPDEs given by Buckdahn and Ma $\cite{BM1,BM2}$ which, in our mind is natural and coincide (if $g\equiv 0$) with the well-know viscosity solution of PDEs introduced by Crandall et $al$ $\cite{Cal}$. In this fact, we will work in the sequel of this paper with the version of backward doubly SDEs introduced in \cite{BM1,BM2}, which is in fact a time reversal of that considered by Pardoux and Peng \cite{PP4}.
Indeed, for $l, f$ be Lipschitz continuous functions in their spatial variables and $g\in C_{b}^{0,2,3}([0,T]\times%
\R^{d}\times\R;%
\R^{d})$, they consider a class of backward doubly SDEs is of this following form:
\begin{eqnarray}
Y_{s}^{t,x}&=&l(X^{t,x}_{0})+\int_{0}^{s}f(r,X^{t,x}_{r},Y_{r}^{t,x},
Z_{r}^{t,x})\;dr+\int_{0}^{s}g(r,X^{t,x}_{r},Y_{r}^{t,x})\;dB_{r}
\notag \\
&&-\int_{0}^{s} Z_{r}^{t,x}\downarrow dW_{r},\ s\in[0,t]. \label{a1bis}
\end{eqnarray}
The diffusion process $X^{t,x}$ is the unique solution of the forward SDE
\begin{eqnarray}
X^{t,x}_{s}=x+\int^{t}_{s}b(r,X^{t,x}_{r})\;dr+\int^{t}_{s}\sigma
(r,X^{t,x}_{r})\downarrow dW_{r}\ s\in[0,t],  \label{a4}
\end{eqnarray}
where $b$ and $\sigma$ are some measurable functions. Here the superscript $(t,x)$ indicates the dependence of the solution on the initial date $(t, x)$, and it will be omitted when the context is clear. Buckdahn and Ma proved in their two works \cite{BM1,BM2}, among other things, that $u(t,x)=Y^{t,x}_t$ is a stochastic viscosity solution of nonlinear parabolic SPDE:
\begin{eqnarray}
\begin{array}{l}
du(t,x)=[\mathcal{L}u(t,x)+f(t,x,u(t,x),(\nabla
u\sigma)(t,x))]\;dt  \\\\
\qquad\qquad+g(t,x,u(t,x))\;dB_{t},\,\, (t,x)\in (0,T)\times\R^{d},\\\\
u(0,x)=l(x),\qquad\qquad\,\,\,x\in\R^{d},
\end{array}
\label{a3}
\end{eqnarray}
where $\mathcal{L}$ defined by
\begin{eqnarray*}
\mathcal{L}=\frac{1}{2}\sum_{i,j=1}^{n}\sum_{l=1}^{k}\sigma_{il}\sigma_{lj}(x)\partial^{2}_{x_{i}x_{j}}
+\sum_{j=1}^{n}b_{j}(x)\partial_{x_{j}},
\end{eqnarray*}
is the infinitesimal operator generated by the diffusion process $X^{t,x}$.
More precisely, they show thank to the Blumenthal $0$-$1$ law that
\begin{eqnarray}
u(t,x)&=&\E\left
\{l(X_{0}^{t,x})+\int_{0}^{t}
f(r,X_{r}^{t,x},Y_{r}^{t,x},Z_{r}^{t,x})\;ds\right.\nonumber\\
&&\left.+\int_{0}^{t}
g(r,X_{r}^{t,x},Y_{r}^{t,x})\;dB_r\mid\mathcal{F}%
^{B}_{t}\right\}.\label{viscosity}
\end{eqnarray}
It is well know that $u$ is a $\mathcal{F}^{B}$-measurable field.
However, to the best of our knowledge, to date there has been no discussion in the literature concerning the path regularity of the process $Z$ when $f$ and $l$ are only Lipschitz continuous, even in the special cases where the coefficients are enough regular.

Our goal in this paper is twofold. First we show that if the coefficients $l$ and $f$ are continuously differentiable, then the viscosity solution
$u$ of the SPDE $(\ref{a3})$ will have a continuous spatial gradient $\partial_x u $ and, more important, the following probabilistic representation holds:
\begin{eqnarray}
\partial_{x}u(t,x)&=&\mbox{I\hspace{-.15em}E}\left\{l(X_{0}^{t,x})\;N^{t}_{0}+
\int_{0}^{s}
f(r,X_{r}^{t,x},Y_{r}^{t,x},Z_{r}^{t,x})N_{r}^{s}\;dr\right.\nonumber\\
&&\left.+\int_{0}^{t}
g(r,X_{r}^{t,x},Y_{r}^{t,x})N_{r}^{s}\;dB_r\mid\mathcal{F}^{B}
_{t}\right\}  \label{a5}
\end{eqnarray}
where $N_{.}^{s}$ is some process defined on $[0,s]$, depending
only on the solutions of the forward SDE $(\ref{a4})$ and its variational equation respectively. This representation can be thought of as a new type of nonlinear Feynman-Kac formula for the derivative of $u$, which does not seem to exist in the literature. The main significance of the formula, however, lies in that it does not depend on the derivatives of the coefficients of the backward doubly BSDE $(\ref{a1bis})$.  Because of this special feature, we can then derive a representation
\begin{eqnarray}
Z^{t,x}_s=\partial_x u(t,X^{t,x}_s)\sigma(t,X^{t,x}_s), \, s\in[0,t],\label{regularity}
\end{eqnarray}
under only a Lipschitz condition on $l$ and $f$. This latter representation then enables us to prove the path regularity of the process $Z$, the second goal of this paper, even in the case where the terminal value of $Y^{t,x}$ is of the form $l(X_{t_0},..., X_{t_n})$, where $\pi:0 = t_0 < .. < t_n= t$ is any partition of $[0, t]$, a result that does not seem to be amendable by any existing method.

Let us recall that this two  representations have already be given by Ma and Zhang \cite{MZ} in the case of a probabilistic representation for solutions of PDEs via BSDEs. Consequently our approach is be inspired by their works. However, there are particularities: first, the derivative notion is take in the flow sense (independent of $\omega_{1}$) because $u$, the stochastic viscosity solution of the SPDE $(\ref{a3})$, is a random field. Secondly, the proof to the continuity of the representation of the process $Z$ need, since ${\bf F}_{s}^{t}=(\mathcal{F}_{s}^{B}\otimes\mathcal{F}^{W}_{s,t})_{0\leq s\leq t}$ is not a filtration, ${\bf G}_{s}^{t}=(\mathcal{F}_{s}^{B}\otimes\mathcal{F}^{W}_{0,t})_{0\leq s\leq t}$ which is a filtration.

The rest of this paper is organized as follows. In section 2 we give all the necessary preliminaries. In section 3 we establish the new Feynman-Kac formula between coupled forward backward doubly SDE $(\ref{a1})$-$(\ref{a4})$ and the SPDE $(\ref{a3})$, under the $C^{1}$-assumption of the coefficients. The section $4$ is devoted to give the main representation theorem assuming only the Lipschitz condition of the coefficients $l$ and $f$. In section $5$ we study the path regularity of the process $Z$.

\section{Preliminaries}
\setcounter{theorem}{0} \setcounter{equation}{0}
Let $T>0$ a fixed time horizon. Throughout this paper $\{W_{t}, 0\leq t\leq T\}$ and $\{B_{t}, 0\leq t\leq T\}$ will denote two independent $d$-dimensional Brownian motions defined on the complete probability spaces $(\Omega_{1},\mathcal{F}_{1},\mbox{I\hspace{-.15em}P}_{1})$ and $(\Omega_{2},\mathcal{F}_{2},\P_{2})$ respectively. For any process $\{U_s,\, 0\leq s\leq T\}$ defined on $(\Omega_i,\mathcal{F}_i,\P_i)\ (i=1,2)$, we write $\mathcal{F}^{U}_{s,t}=\sigma(U_r-U_s,\, s\leq r\leq t)$ and $\mathcal{F}^{U}_{t}=\mathcal{F}^{U}_{0,t}$. Unless otherwise specified we consider
\begin{eqnarray*}
\Omega=\Omega_{1}\times\Omega_{2},\,\,\mathcal{F}=\mathcal{F}_{1}\otimes\mathcal{F}_{2}\,\,
\mbox{ and}\,\, \P=\P_{1}\otimes\P_{2}.
\end{eqnarray*}
In addition, we put for each $t\in[0,T]$,
$${\bf F}=\{\mathcal{F}_{s}=\mathcal{F}_{s}^{B}\otimes\mathcal{F}^{W}_{s,T}\vee\mathcal{N},\; 0\leq s\leq T\}$$
where $\mathcal{N}$ is the collection of $\P$-null sets. In other words, the collection ${\bf F}$ is
$\P$-complete but is neither increasing nor decreasing so that, it is not a filtration. Let us tell also that random variables $\xi(\omega_{1}),\,\omega_{1}\in \Omega_{1}$ and $\zeta(\omega_{2}),\, \omega_{2}\in
\Omega_{2}$ are considered as random variables on $\Omega $ via
the following identification:
\begin{eqnarray*}
\xi(\omega_{1},\omega_{2})=\xi(\omega_{1});\,\,\,\,\,\zeta(\omega_{1},\omega_{2})=\zeta(\omega_{2}).
\end{eqnarray*}
Let $E$ denote a generic Euclidean space; and regardless
of its dimension we denote $\langle ; \rangle$	 to be the inner product and $|.|$ the norm in $E$. if
an other Euclidean spaces are needed, we shall label them as $E_1; E_2,.,.,$ etc. Furthermore, we use the notation $\partial_t = \frac{\partial}{\partial t}, \ \partial_x = (\frac{\partial}{\partial x_1},\frac{\partial}{\partial x_2},..,\frac{\partial}{\partial x_d})$ and $\partial^{2}=\partial_{xx}=(\partial^{2}_{x_{i}x_j})^{d}_{i,j=1}$, for $(t,x) \in[0, T]\times\R^d$. Note that if $\psi = (\psi^1,..,\psi^d):\R^d \rightarrow\R^d$, then $\partial_{x}\psi\triangleq(\partial x_j\psi^i)^{d}_{i,j=1}$ is a matrix. The meaning of $\partial_{xy}, \partial_{yy}$, etc. should be clear from the context.

The following spaces will be used frequently in the sequel (let $\mathcal{X}$ denote a generic Banach space):
\begin{enumerate}
\item For $t\in [0,T], L^{0}([0,t];\mathcal{X})$ is the space of all measurable
functions $\varphi:[0,t]\mapsto \mathcal{X}$.
\item For $0\leq t\leq T, C([0,t];\mathcal{X})$ is the space of all continuous
functions $\varphi:[0,t]\mapsto \mathcal{X}$; further, for any $p>0$ we denote $%
\displaystyle{|\varphi|^{*,p}_{0,t}=\sup_{0\leq s\leq t}
\|\varphi(s)\|^{p}_{\mathcal{X}}}$ when the context is clear.
\item For any $k,\,n\geq 0,\; C^{k,n}([0,T]\times E;E_{1})$
is the space of all $E_{1}$-valued functions $\varphi(t,e),\; (t,e) \in [0,T]\times E$, such that they are $k$-times
continuously differentiable in $t$ and $n$-times continuously differentiable in $e$.
\item $C_{b}^{1}([0,T]\times E;E_{1})$ is the space of those $%
\varphi \in C^{1}([0,T]\times E;E_{1})$ such that all the partial
derivatives are uniformly bounded.
\item $W^{1,\infty}(E,E_{1})$ is the space of all
measurable functions $\psi: E\mapsto E_{1}$, such that for
some constant $K>0$ it holds that $|\psi(x)-\psi(y)|_{E_{1}}%
\leq K|x-y|_{E}, \forall x, y \in E$.
\item For any sub-$\sigma$-field $\mathcal{G} \subseteq \mathcal{F}_{T}^{B}$
and $0\leq p< \infty,\, L^{p}(\mathcal{G};E)$ denote all $E$-valued $\mathcal{G}$-measurable random variable $\xi$ such that $%
\mbox{I\hspace{-.15em}E}|\xi|^{p}<\infty$. Moreover, $\xi\in L^{\infty}(%
\mathcal{G};E)$ means it is $\mathcal{G}$-measurable and bounded.
\item For $0\leq p <\infty, L^{p}(\mathbf{F},[0,T];\mathcal{X})$ is the space
of all $\mathcal{X}$-valued, $\mathbf{F}$-adapted processes $\xi$ satisfying $%
\displaystyle{\ \mbox{I\hspace{-.15em}E}\left(\int_{0}^{T}\|\xi_{t}\|^{p}_{%
\mathcal{X}}dt\right)<\infty}$; and also, $\xi \in L^{\infty}(\mathbf{F},[0,T];%
\mbox{I\hspace{-.15em}R}^{d})$ means that the process $\xi$ is uniformly essentially bounded in $%
(t,\omega)$.
\item $C(\mathbf{F},[0,T]\times E; E_{1})$ is the
space of $E_{1}$-valued, continuous random field $\varphi :\Omega\times[%
0,T]\times E$, such that for fixed $e\in E$, $\varphi(.,.,e)
$ is an $\mathbf{F}$-adapted process.
\end{enumerate}
To simplify notation we often write $C([0, T] \times E;E_1) = C^{0,0}([0, T] \times E;E_1)$; and if $E_1 = \R$, then
we often suppress $E_1$ for simplicity (e.g., $C^{k,n}([0, T]\times E;\R)
= C^{k,n}([0, T]\times E),\; C^{k,n}({\bf F},\ [0, T]\times E;\R) = C^{k,n}({\bf F},\ [0, T]\times E), ...,$ etc.). Finally, unless otherwise specified (such as process $Z$ mentioned in Section 1), all vectors in the paper will be regarded as column vectors.

Throughout this paper we shall make use of the following standing assumptions:
\begin{description}
\item  $(\mathbf{A1})$ The functions $\sigma \in
C_{b}^{0,1}([0,T]\times\R^{d};\R
^{d\times d}),\, b\in C_{b}^{0,1}([0,T]\times\R^{d};%
\R^{d})$; and all the partial derivatives of $b$ and $%
\sigma$ (with respect to $x$) are uniformly bounded by a common
constant $K>0 $. Further, there exists constant $c>0$, such that
\begin{eqnarray}
\xi^{T}\sigma(t,x)\sigma(t,x)^{T}\xi\geq c|\xi|^{2}, \forall x, \xi \in %
\R^{d}, t\in [0,T].  \label{b1}
\end{eqnarray}

\item  $(\mathbf{A2})$ The function $f\in C(\mathcal{F}^{B},[0,T]\times \R^{d}\times\R\times\R^{d})\cap W^{1,\infty}([0,T]\times\R
^{d}\times\R\times\R^{d})$ and $l\in W^{1,\infty}(\R^{d})$.
Furthermore, we denote the Lipschitz constants of $f$ and $l$ by a common one $K>0$ as
in $(\mathbf{A1})$; and we assume that
\begin{eqnarray}
\sup_{0\leq t \leq
T}\left\{|b(t,0)|+|\sigma(t,0)|+|f(t,0,0,0)|+|g(0)|\right\}\leq K.
\label{b2}
\end{eqnarray}

\item  $(\mathbf{A3})$ The function $g \in C_{b}^{0,2,3}([0,T]\times%
\R^{d}\times\R;%
\R^{d})$
\end{description}

The following results are either standard or
slight variations of the well-know results in SDE and backward doubly SDE literature; we give only the statement for ready reference.
\begin{lemma}
Suppose that $b\in C(\mathbf{F},[0,T]\times \R^{d};\R^{d})\cap L^{0}(\mathbf{F}%
,[0,T]; W^{1,\infty}(R^{d};\R^{d})),\newline
\sigma\in C(\mathbf{F},[0,T]\times \R^{d};%
\R^{d\times d})\cap L^{0}(\mathbf{F},[0,T];
W^{1,\infty}(R^{d};\R^{d\times d}))$, with a common
Lipschitz constant $K >0$. Suppose also that $b(t,0)\in L^{2}(\mathbf{F},[0,T];
\R^{d})$ and $\sigma(t,0)\in L^{2}(\mathbf{F},[0,T];
\R^{d\times d})$. Let $X$ be the unique solution of
the following forward SDE
\begin{eqnarray}
X_{s}=x+\int^{t}_{s}b(r,X_{r})\;dr+\int^{t}_{s}
\sigma(r,X_{r})\;dW_{r}.  \label{b4}
\end{eqnarray}
Then for any $p\geq 2$, there exists a constant $C>0$ depending only
on $p, T $ and $K$, such that
\begin{eqnarray}
E(|X|_{0,t}^{*,p})\leq C\left\{|x|^{p}+%
\E\int_{0}^{T}[|b(s,0)|^{p}+|\sigma(s,0)|^{p}]\;ds%
\right\}  \label{b5}
\end{eqnarray}
\end{lemma}

\begin{lemma}
Assume $f\in C(\mathbf{F},[0,T]\times \R \times \R^{d})\cap L^{0}(\mathbf{F}%
,[0,T]; W^{1,\infty}(\R \times R^{d}))$, with a
uniform Lipschitz constant $K>0$, such that $f(s,0,0)\in L^{2}(\mathbf{F},[0,T])$\,\, %
and $g\in C(\mathbf{F},[0,T]\times \R \times \R^{d};\R^{d})\cap L^{0}(\mathbf{F},[0,T]; W^{1,\infty}(
\R \times R^{d};\R^{l}))$ with a common uniform
Lipschitz constant $K>0$ with respect the first variable and the Lipschitz constant
$0<\alpha<1$  which respect the second variable and such
that $g(s,0,0)\in L^{2}(\mathbf{F},[0,T]) $. For any $\xi \in L^{2}({%
\mathcal{F}}_{0};\R)$, let $(Y,Z)$ be the adapted solution to the BDSDE:
\begin{eqnarray}
Y_{s}=\xi+\int_{0}^{s}f(r,Y_{r},Z_{r})\;dr+\int_{0}^{s}
g(r,Y_{r},Z_{r})\;dB_{r}-\int_{0}^{s}Z_{r}\downarrow dW_{r}.
\label{b6}
\end{eqnarray}
Then there exists a constant $C>0$ depending only on $T$ and on the Lipschitz constants $K$ and $\alpha$, such that
\begin{eqnarray}
\E\int_{0}^{T}|Z_{s}|^{2}ds \leq C\E\left\{|\xi|^{2}+\int_{0}^{T}[|f(s,0,0)|^{2}
+|g(s,0,0)|^{2}]\;ds\right\}.  \label{b7}
\end{eqnarray}
Moreover, for all $p\geq 2$, there exists a constant $C_{p}>0$, such that
\begin{eqnarray}
\E(|Y|_{0,t}^{*,p})\leq C_{p}\E
\left\{|\xi|^{p}+\int_{0}^{T}[|f(s,0,0)|^{p}+|g(s,0,0)|^{p}]\;ds\right\}
\label{b8}
\end{eqnarray}
\end{lemma}
We now review some basic facts of the anticipating stochastic calculus, especially those related to the backward doubly SDEs (see Pardoux-Peng $\cite{PP4}$).
For any random variables $\xi$ of the form
\begin{eqnarray*}
\xi=F\left(\int_{0}^{T}\varphi_{1}dW_{t},..,\int_{0}^{T}\varphi_{n}dW_{s}
;\int_{0}^{T}\psi_{1}dB_{s},...,\int_{0}^{T}\psi_{p}dB_{s}\right)
\end{eqnarray*}
with $F\in C_{b}^{\infty}(\mbox{I\hspace{-.15em}R}^{n+p}),\,
\varphi_{1},...,\varphi_{n}\in L^{2}([0,T],\mbox{I\hspace{-.15em}R}^{d}),\,\psi_{1},...,\psi_{n}\in L^{2}([0,T],\mbox{I\hspace{-.15em}R}%
^{d}),$ we let
\begin{eqnarray*}
D_{t}\xi=\sum_{i=}^{n}\frac{\partial F}{\partial x_{i}}\left(\int_{0}^{T}%
\varphi_{1}dW_{t},..,\int_{0}^{T}\varphi_{n}dW_{s}
;\int_{0}^{T}\psi_{1}dB_{s},...,\int_{0}^{T}\psi_{p}dB_{s}\right)\varphi_{i}(t).
\end{eqnarray*}
For such a $\xi$, we define its $1,2$-norm as:
\begin{eqnarray*}
\|\xi\|_{1,2}^{2}=\mbox{I\hspace{-.15em}E}\left[|\xi|^{2}+\mbox{I\hspace{-.15em}E}%
\int_{0}^{T}|D_{r}\xi|^{2}dr\right].
\end{eqnarray*}
$\mathcal{S}$ denoting the set of random variable of the above form, we define the Sobolev space
$$\mbox{I\hspace{-.15em}D}^{1,2}\triangleq\overline{\mathcal{S}}^{\|.\|_{1,2}}.$$
The "derivation operator" $D_.$ extends as an operator from $\D^{1,2}$ into $L^{2}(\Omega,L^{2}([0,T],\R^{d}))$.

We shall apply the previous anticipative calculus to the  coupled forward backward doubly SDEs $(\ref{a4})$-$(\ref{a1bis})$.
In this fact, let us consider the following variational equation that will play a important role in this paper: for $i=1,..,d$,
\begin{eqnarray}
\nabla_{i} X_{s}^{t,x}&=&e_{i}+\int^{t}_{s}\partial_{x}b(r,X_{r}^{t,x})\nabla_{i}
X_{r}^{t,x}dr+\sum_{j=1}^{d}\int^{t}_{s}\partial_{x}\sigma^{j}
(r,X_{r}^{t,x})\nabla_{i}X_{r}^{t,x}\downarrow dW_{r}^{j}, \nonumber\\
\nabla_{i} Y_{s}^{t,x}&=&\partial_{x}l(X_{0}^{t,x})\nabla_{i}
X_{0}^{t,x}\nonumber\\
&&+\int^{s}_{0}[\partial_{x}f(r,\Xi^{t,x}(r))\nabla_{i}
X_{r}^{t,x}+\partial_{y}f(r,\Xi^{t,x}(r))\nabla_{i} Y_{r}^{t,x}+\langle\partial_{z}f(r,\Xi^{t,x}(r)),\nabla_{i}
Z_{r}^{t,x}\rangle]dr\nonumber\\
&&+\int^{s}_{0}[\partial_{x}g(r,\Theta^{t,x}(r))\nabla_{i}
X_{r}^{t,x}+\partial_{y}g(r,\Theta^{t,x}(r))\nabla_{i} Y_{r}^{t,x}]dB_{r}-\int^{s}_{0}\nabla_{i}Z_{r}^{t,x}\downarrow dW_{r},  \label{b10}
\end{eqnarray}
where $e_{i}=(0,...,\overset{i}{1},...,0)^{T}\in\R^{d},\Xi^{t,x}=(\Theta^{t,x},Z^{t,x}),\,\Theta^{t,x}=(X^{t,x},Y^{t,x})$ and $\sigma^{j}(.)$
is the $j$-th column of the matrix $\sigma(.)$.
We recall again that the superscription $^{t,x}$ indicates the dependence of the solution on the initial
date $(t,x)$, and will be omitted when the context is clear. We also remark that under the above assumptions, $$\left(\nabla X^{t,x},\nabla Y^{t,x},\nabla Z^{t,x} \right)\in L^{2}( \mathbf{F};C([0,T];
\R^{d\times d})\times C([0,T];\R^{d})\times L^{2}([0,T];\R^{d\times d})).$$ Further the $d\times d$-matrix-valued process $\nabla X^{t,x}$
satisfies a linear SDE and $\nabla X_{t}^{t,x}=I$, so that $[\nabla X_{s}^{t,x}]^{-1}$ exists for $%
s\in [0,t], \P$-a.s. and we have the following:
\begin{lemma}
Assume that $\left(\mathbf{A1}\right)$ holds; and suppose that $f\in
C^{0,1}_{b}([0,T]\times\mbox{I\hspace{-.15em}R}^{2d+1})$ and\newline
$g\in C^{0,2,3}_{b}([0,T]\times\mbox{I\hspace{-.15em}R}^{d+1};%
\mbox{I\hspace{-.15em}R}^{d})$. Then $(X,Y,Z)\in L^{2}([0,T];%
\mbox{I\hspace{-.15em}D}^{1,2}(\mbox{I\hspace{-.15em}R}^{2d+1}))$,
and there exists a version of $(D_{s}X_{r},D_{s}Y_{r},D_{s}Z_{r})$
that satisfies
\begin{eqnarray}
\left\{
\begin{array}{l}
D_{s}X_{r}=\nabla X_{r}(\nabla
X_{s})^{-1}\sigma(s,X_{s})\mathbf{1}_{\{s\leq
r\}}, \\
D_{s}Y_{r}=\nabla Y_{r}(\nabla
X_{s})^{-1}\sigma(s,X_{s})\mathbf{1}_{\{s\leq
r\}}, \\
D_{s}Z_{r}=\nabla Z_{r}(\nabla
X_{s})^{-1}\sigma(s,X_{s})\mathbf{1}_{\{s\leq r\}},
\end{array}
\right.\,\,\ 0\leq s,r \leq t.  \label{b11}
\end{eqnarray}
\end{lemma}
\begin{lemma}
Suppose that $F\in \mbox{I\hspace{-.15em}D}^{1,2}$. Then

\begin{description}
\item  $(i)$(Integration by parts formula): for any $u \in Dom(\delta)$ such
that $Fu\in L^{2}([0,T]\times\Omega;\mbox{I\hspace{-.15em}R}^{d})$, one has $%
Fu \in Dom(\delta)$, and it holds that
\begin{eqnarray*}
\int_{0}^{T}\langle Fu_{t},dW_{t}\rangle=\delta(Fu)=F\int_{0}^{T}\langle
u_{t},dW_{t}\rangle-\int_{0}^{T}D_{t}Fu_{t}dt;
\end{eqnarray*}

\item  $(ii)$(Clark-Hausman-Ocone formula):
\begin{eqnarray*}
F=\mbox{I\hspace{-.15em}E}(F)+\int_{0}^{T}\mbox{I\hspace{-.15em}E}\{D_{t}F\mid
\mathcal{F}_{t}\}dW_{t}.
\end{eqnarray*}
\end{description}
\end{lemma}

\section{Relations to stochastic PDE revisited}
\setcounter{theorem}{0} \setcounter{equation}{0}
In this section we prove the relation (\ref{regularity}) between the forward
backward doubly SDE $(\ref{a1bis})$-$(\ref{a4})$ and the quasi-linear SPDE $(\ref{a3})$, under the condition that the coefficients are only continuously differentiable. Indeed, since Buckdahn and Ma $\cite{BM1,BM2}$ provide that, if $f$ and $l$ are only Lipschitz continuous, the quantity $\displaystyle{u(t,x)=Y_{t}^{t,x}}$ is a stochastic viscosity solution of the quasi-linear SPDE $(\ref{a3})$, relation in (\ref{regularity}) becomes questionable.  Our objective is to fill this gap in the  literature and to extend the results of Ma and Zhang $\cite{MZ}$ given in the case of the probabilistic interpretation of PDEs via the BSDEs.
\begin{theorem}
\label{T1}Assume $\left(\mathbf{A1}\right)$ and
$\left(\mathbf{A3}\right)$
and suppose that $f\in C_{b}^{0,1}([0,T]\times \mbox{I\hspace{-.15em}R}^{d}\times %
\mbox{I\hspace{-.15em}R}\times\mbox{I\hspace{-.15em}R}^{d})$ and
\,$l\in C_{b}^{1}(\mbox{I\hspace{-.15em}R}^{d})$. Let $%
(X^{t,x},Y^{t,x},Z^{t,x})$ be the adapted solution to the FBDSDE $(\ref{a1bis})$-$(\ref{a4})$,
and set $u(t,x)=Y_{t}^{t,x}$ the stochastic viscosity of SPDE $(\ref{a3})$. Then,
\item  $\left(i\right)$ $\partial_{x}u(t,x)$ exists for all $(t,x)\in
[0,T]\times \mbox{I\hspace{-.15em}R}^{d}$; and for each $(t,x)$ and
i=1,...,d, the following representation holds:
\begin{eqnarray}
\partial_{x_{i}}u(t,x)&=&\mbox{I\hspace{-.15em}E}\left\{%
\partial_{x}l(X_{0}^{t,x})\nabla_{i}X_{0}^{t,x}\right.\nonumber\\
&&\left.+\int^{t}_{0}[\partial_{x}f(r,\Xi^{t,x}(r))\nabla_{i}X_{r}^{t,x}+
\partial_{y}f(r,\Xi^{t,x}(r))\nabla_{i}Y_{r}^{t,x}+\partial_{z}f(r,\Xi^{t,x}(r))\nabla_{i}Z_{r}^{t,x}]dr\right.\nonumber\\
&&\left.+\int^{t}_{0}[\partial_{x}g(r,\Theta^{t,x}(r))\nabla_{i}X_{r}^{t,x}+
\partial_{y}g(r,\Theta^{t,x}(r))\nabla_{i}Y_{r}^{t,x}]dB_r\mid \mathcal{F}%
^{B}_{t}\right\}  \label{c1}
\end{eqnarray}
where $\Theta^{t,x}=(X^{t,x},Y^{t,x}),\;\Xi^{t,x}=(\Theta^{t,x},Z^{t,x})$,
and $(\nabla X^{t,x}, \nabla Y^{t,x}, \nabla Z^{t,x})$ the unique solution of equation $(\ref{b10})$;
\item  $\left(ii\right)$ $\partial_{x}u(t,x)$ is continuous on $[0,T]\times %
\mbox{I\hspace{-.15em}R}^{d}$;
\item  $\left(iii\right)$ $\displaystyle{Z_{s}^{t,x}=\partial_{x}
u(s,X^{t,x}_{s})\sigma(s,X_{s}^{t,x}), \forall\ s\in [0,t], \mbox{I%
\hspace{-.15em}P}}$-a.s.
\end{theorem}
\begin{proof}
For the simple presentation we take $d=1$. The higher dimensional case can be treated
in the same way without substantial difficulty. We use the simpler
notations $l_{x}, (f_{x},f_{y},f_{z}), (g_{x},g_{y},g_{z})$ respectively for the
partial derivatives of $l, f $ and $g$.\\
The proof is inspired by the approach of Ma and Zhang $\cite{MZ}$ (see Theorem $3.1$). Nevertheless, there exists slight difference due in the fact that the solution of SPDE's is a random field; more precisely will show that it is a conditional expectation with respect the filtration $({\mathcal{F}}^{B}_{t})_{0\leq t\leq T}$.

We first prove $(i)$. Let $(t,x)\in[0,T]\times\R$ be fixed. For $h\neq 0$, we define:
\begin{eqnarray*}
\nabla X_{s}^{h}=\frac{X_{s}^{t,x+h}-X_{s}^{t,x}}{h}; \nabla Y_{s}^{h}=\frac{%
Y_{s}^{t,x+h}-Y_{s}^{t,x}}{h}; \nabla Z_{s}^{h}=\frac{
Z_{s}^{t,x+h}-Z_{s}^{t,x}}{h}\;\;\; s\in [0,t].
\end{eqnarray*}
It follows analogously of the proof of Theorem 2.1 in \cite{PP4}) that
\begin{eqnarray}
\E\{|\Delta Y^{h}|^{*,2}_{0,t}=\E\{|\nabla Y^{h}-\nabla
Y^{t,x}|^{*,2}_{0,t}\}\rightarrow 0 \ as\ h \rightarrow 0.  \label{cp4}
\end{eqnarray}
We know also that processes $Y^{t,x},\;Y^{t,x+h},\;\nabla Y^{h}$ and $\Delta Y^{h}$ are all adapted to the $\sigma$-algebra ${\bf F}^{t}=(\mathcal{F}^{t}_s)_{0\leq s\leq t}$, where $\mathcal{F}^{t}_s=\mathcal{F}^{B}_{s}\otimes\mathcal{F}^{W}_{s,t}$. In particular, since $W$ is a Brownian motion on $(\Omega_2, \mathcal{F}_2,\P_2)$, applying the Blumenthal $0$-$1$ law (see, e.g, \cite{KS}), $Y^{t,x}_{t}=u(t,x),\;Y^{t,x+h}_{t}=u(t,x+h),\;\nabla Y^{h}_{t}=\frac{1}{h}[u(t,x+h)-u(t,x)]$ and $\Delta Y^{h}_{t}$ are all independent of (or a constant with respect to) $\omega_{2}\in\Omega_2$. Therefore we conclude from the above that $\partial_{x}u$ exist, as the random field and $\partial_{x}u(t,x)=\nabla Y^{t,x}_t$, for all $(t,x)$.
Finally, taking the conditional expectation on the both sides of (\ref{b10}) at $s=t$, the representation (\ref{c1}) hold and finish the prove of $(i)$.

We now prove (ii). Let $(t_i,x_i)\in[0,T]\times\R,\, i=1,2$. Knowing that $t_1$ and $t_2$ played inverse roles one another, we assume without losing a generality that $t_1 <t_2$. Since $\partial_{x} u$ is a conditional expectation with respect the filtration $(\mathcal{F}^{B}_s)_{0\leq s\leq t}$, we have
\begin{eqnarray}
|\partial_{x}u(t_1,x_1)-\partial_{x}u(t_2,x_2)|&\leq&\E\{A(t_1,x_1)-A(t_2,x_2) \mid\mathcal{F}^{B}_{t_1}\}\nonumber\\
&&+\left|\E\{A(t_2,x_2)\mid\mathcal{F}^{B}_{t_1}\}-\E\{A(t_2,x_2)\mid\mathcal{F}^{B}_{t_2}\}\right|,
\label{o1}
\end{eqnarray}
where
\begin{eqnarray}
A(t,x)&=&l_x(X_{0}^{t,x})\nabla X_{0}^{t,x}\nonumber\\
&&+\int_{0}^{t }[f_x(r,\Xi^{t,x}(r))\nabla X_{r}^{t,x}+
f_y(r,\Xi^{t,x}(r))\nabla Y_{r}^{t,x}+f_z(r,\Xi^{t,x}(r))\nabla Z_{r}^{t,x}]\,dr\nonumber\\
&&+\int^{t}_{0}[g_x(r,\Theta^{t,x}(r))\nabla X_{r}^{t,x}+
g_y(r,\Theta^{t,x}(r))\nabla Y_{r}^{t,x}]dB_r.\label{q}
\end{eqnarray}
 Thanks to the quasi-left-continuity of $(\mathcal{F}^{B}_s)_{0\leq s\leq t}$, we see that
 \begin{eqnarray}
\lim_{t_{1}\downarrow t_{2}}\left|\E(A(t_2,x_2) \mid\mathcal{F}^{B}_{t_1})-\E(A(t_2,x_2)\mid\mathcal{F}^{B}_{t_2})\right|=0,\label{leftcont}
\end{eqnarray}
independently of $x_{2}$. In virtue of $(\ref{o1})$ and $(\ref{leftcont})$), to prove $(ii)$ it remain to show that
\begin{eqnarray}
\lim_{t_{1}\downarrow t_{2}x_{1}\rightarrow x_{2}}\E\{A(t_1,x_1)-A(t_2,x_2) \mid\mathcal{F}^{B}_{t_1}\}=0.\label{cont1}
\end{eqnarray}
To this end, since $A(t,x)$ is a stochastic process and in virtue of Kolmogorov-Centsov Theorem (see \cite{KS}), it suffices to show that
\begin{eqnarray*}
\E\left(|A(t_1,x_1)-A(t_2,x_2)|^{2}\right)\leq C(|t_1-t_2|^{2}+|x_1-x_2|^{2}),
\end{eqnarray*}
what we do now. Recalling the definition of $A(t_i,x_i),\, i=1,2$ and denoting
$$G^{t,x}(r)=f_x(r,\Xi^{t,x}(r))\nabla X_{r}^{t,x}+
f_y(r,\Xi^{t,x}(r))\nabla Y_{r}^{t,x}+f_z(r,\Xi^{t,x}(r))\nabla Z_{r}^{t,x}$$
and
$$H^{t,x}(r)=g_x(r,\Theta^{t,x}(r))\nabla X_{r}^{t,x}+
g_y(r,\Theta^{t,x}(r))\nabla Y_{r}^{t,x},$$ we get
\begin{eqnarray*}
|A(t_1,x_1)-A(t_2,x_2)|&\leq& |l_x(X^{t_1,x_1}_{0})\nabla X^{t_1,x_1}_{0}-l_x(X^{t_2,x_2}_{0})\nabla X^{t_2,x_2}_{0}|\\
&&+\int_{0}^{t_1}|G^{t_1,x_1}(r)-G^{t_2,x_2}(r)|dr+\left|\int_{0}^{t_1}(H^{t_1,x_1}(r)-H^{t_2,x_2}(r))dB_r\right|\\
&&+\int_{t_1}^{t_2}|G^{t_2,x_2}(r)|dr+\left|\int_{t_1}^{t_2}H^{t_2,x_2}(r)dB_r\right|.
\end{eqnarray*}
Taking the expectation, it follows by Hölder's and Burkölder-Gundy Davis inequalities that
\begin{eqnarray*}
\E\left(|A(t_1,x_1)-A(t_2,x_2)|^{2}\right)&\leq& C\E\left\{|l_x(X^{t_1,x_1}_{0})\nabla X^{t_1,x_1}_{0}-l_x(X^{t_2,x_2}_{0})\nabla X^{t_2,x_2}_{0}|^{2}\right.\\
&&+\left.\int_{0}^{t_1}|G^{t_1,x_1}(r)-G^{t_2,x_2}(r)|^{2}dr+\left|\int_{0}^{t_1}(H^{t_1,x_1}(r)-H^{t_2,x_2}(r))dB_r\right|^{2}\right.\\
&&+\left.\int_{t_1}^{t_2}|G^{t_2,x_2}(r)|^{2}dr+\left|\int_{t_1}^{t_2}H^{t_2,x_2}(r)dB_r\right|^{2}\right\}\\
&\leq&C\E\left\{|l_x(X^{t_1,x_1}_{0})\nabla X^{t_1,x_1}_{0}-l_x(X^{t_2,x_2}_{0})\nabla X^{t_2,x_2}_{0}|^{2}\right.\\
&&+\left.\int_{0}^{t_1}|G^{t_1,x_1}(r)-G^{t_2,x_2}(r)|^{2}dr+\int_{0}^{t_1}|H^{t_1,x_1}(r)-H^{t_2,x_2}(r)|^{2}dr\right.\\
&&+\left.(t_2-t_1)\int_{t_1}^{t_2}|G^{t_2,x_2}(r)|^{2}dr+(t_2-t_1)\int_{t_1}^{t_2}|H^{t_2,x_2}(r)|^{2}dr\right\}.
\end{eqnarray*}
By similar standard computations in Ma and Zhang \cite{MZ} (see proof of Theorem 3.1), we obtain
\begin{eqnarray*}
\E\left(|A(t_1,x_1)-A(t_2,x_2)|^{2}\right)\leq C(|t_2-t_1|^{2}+|x_2-x_1|^2)
\end{eqnarray*}
that provide the proof of $(ii)$.\newline
It remains to prove $(iii)$. For a continuous function $\varphi$, let
us consider $\{\varphi^{\varepsilon}\}_{\varepsilon>0}$ a family of $%
C^{0,\infty}$ functions that converges to $\varphi$ uniformly. Since $%
b,\ \sigma,\ l,\ f$ are all uniformly Lipschitz continuous, we
may assume that the first order partial derivatives of $b^{\varepsilon},\ \sigma^{%
\varepsilon},\ l^{\varepsilon},\ f^{\varepsilon}$
are all uniformly bounded, by the corresponding Lipschitz constants of $b,\sigma,l,f$ uniformly in $\varepsilon>0$. Now we consider the family of FBDSDEs parameterized by $\varepsilon >0$:
\begin{eqnarray}
\left\{
\begin{array}{l}
X^{t,x}_{s}=x+\int^{t}_{s}b^{\varepsilon}(r,X^{t,x}_{r})dr+\int^{t}_{s}
\sigma^{\varepsilon} (r,X^{t,x}_{r})\downarrow dW_{r}; \\
\\
Y_{s}^{t,x}=l^{\varepsilon}(X_{0}^{t,x})+\int_{0}^{s}f^{%
\varepsilon}(r,X_{r}^{t,x},Y_{r}^{t,x},
Z_{r}^{t,x})dr+\int_{0}^{s}g(r,X_{r}^{t,x},Y_{r}^{t,x})dB_{r}-\int_{0}^{s}Z_{r}^{t,x}\downarrow dW_{r}
\end{array}
\right.  \label{pc13}
\end{eqnarray}
and denote it solution by $(X^{t,x}({\varepsilon}),Y^{t,x}({\varepsilon}%
),Z^{t,x}({\varepsilon}))$. We define $\displaystyle{u^{%
\varepsilon}(t,x)=Y_{t}^{t,x}({\varepsilon})}$. Theorem $3.2$ of $
\cite{PP4}$ provide that $u^{\varepsilon}$ is the classical solution of
stochastic PDE
\begin{eqnarray}
\begin{array}{l}
du^{\varepsilon}(t,x)=[\mathcal{L}^{\varepsilon}u(t,x)+f^{\varepsilon}(t,x,u^{\varepsilon}(t,x),(\nabla
u^{\varepsilon}\sigma^{\varepsilon})(t,x))]\;dt  \\\\
\qquad\qquad+g(t,x,u^{\varepsilon}(t,x))\;dB_{t},\,\, (t,x)\in (0,T)\times\R^{d},\\\\
u^{\varepsilon}(0,x)=l^{\varepsilon}(x),\qquad\qquad\,\,\,x\in\R^{d}.
\end{array}
\label{pc14}
\end{eqnarray}
For any $\{x^{\varepsilon}\}\subset\mbox{I\hspace{-.15em}R}^{n}$
such that $x^{\varepsilon}\rightarrow x$ as $\varepsilon \rightarrow 0$,
define \newline
$(X^{\varepsilon},Y^{\varepsilon},Z^{\varepsilon})=(X^{t,x^{\varepsilon}}(%
\varepsilon)
,Y^{t,x^{\varepsilon}}(\varepsilon),Z^{t,x^{\varepsilon}}(\varepsilon))$.
Then it is well know according the work of Pardoux and Peng \cite{PP4} that
\begin{eqnarray}
Y^{\varepsilon}_{s}=u^{\varepsilon}(s,X^{\varepsilon}_{s});\,\,\,\,\
Z^{\varepsilon}_{s}=\partial_{x}u^{\varepsilon}(s,X^{\varepsilon}_{s})%
\sigma^{\varepsilon} (s,X^{\varepsilon}_{s}),\,\,\,\ \forall\ s\in[0,t],
\,\, \mbox{I\hspace{-.15em}P}-a.s.  \label{pc15}
\end{eqnarray}
Now by Lemma $2.1$ and Lemma $2.2$, for all $p\geq 2$ it hold that
\begin{eqnarray}
\mbox{I\hspace{-.15em}E}\left\{|X^{\varepsilon}-X|^{*,p}_{0,t}+|Y^{%
\varepsilon}-Y|^{*,p}_{0,t}+
\int_{0}^{t}|Z^{\varepsilon}_{s}-Z_{s}|^{2}ds\right\}\rightarrow 0
\label{pc16}
\end{eqnarray}
as $\varepsilon \rightarrow 0$. Moreover let us recall $(\nabla
X^{\varepsilon},\nabla Y^{\varepsilon},\nabla Z^{\varepsilon})$ the unique solution of the
variational equation of $(\ref{pc13})$. Using again Lemma $2.1$ and Lemma $2.2$ we get
\begin{eqnarray}
\mbox{I\hspace{-.15em}E}\left\{|\nabla X^{\varepsilon}-\nabla
X|^{*,p}_{0,t}+|\nabla Y^{\varepsilon}-\nabla
Y|^{*,p}_{0,t}+\int_{0}^{t}|\nabla Z^{\varepsilon}_{s}-\nabla
Z_{s}|^{2}ds\right\}\rightarrow 0,  \label{pc18}
\end{eqnarray}
as $\varepsilon \rightarrow 0$.
Thus it is readily seen that
\begin{eqnarray*}
\E\{l^{\varepsilon}_{x}(X_{0}^{\varepsilon})\nabla X^{\varepsilon}_{0}|\mathcal{F}^{B}_t\}\rightarrow\E\{l_{x}(X_{0})\nabla X_{0}|\mathcal{F}^{B}_t\},
\end{eqnarray*}
$\P$-a.s., as $\varepsilon\rightarrow 0$. Furthermore, by the analogue step used in \cite{MZ}, one can show that
\begin{eqnarray*}
\E\left\{\int_{0}^{t}[f^{\varepsilon}_{x}(r)\nabla X^{\varepsilon}_{r}+f^{\varepsilon}_{y}(r)\nabla
Y^{\varepsilon}_{r}+f^{\varepsilon}_{z}(r)\nabla Z^{\varepsilon}_{r}]dr+\int_{0}^{t}[g_{x}(r)\nabla X^{\varepsilon}_{r}+g_{y}(r)\nabla
Y^{\varepsilon}_{r}]dB_r|\mathcal{F}^{B}_t\right\}
\end{eqnarray*}
converge to
\begin{eqnarray*}
\E\left\{\int_{0}^{t}[f_{x}(r)\nabla X_{r}+f_{y}(r)\nabla Y_{r}+f_{z}(r)\nabla Z_{r}]dr+\int_{0}^{t}[g_{x}(r)\nabla X_{r}+g_{y}(r)\nabla Y_{r}]dB_r|\mathcal{F}^{B}_t\right\}
\end{eqnarray*}
$\P$-a.s., as $\varepsilon\rightarrow 0$.
Therefore, we get
\begin{eqnarray*}
\partial_{x} u^{\varepsilon}(t,x^{\varepsilon})\rightarrow \partial_{x} u(t,x),\,\,\,\mbox\,\,{as} \,\,\,\varepsilon \rightarrow 0\,\,\ \P-a.s.,
\end{eqnarray*}
for each fixed $(t,x)\in[0, T]\times\R$.
Consequently, possibly along a subsequence, we obtain
\begin{eqnarray*}
Z_{s}^{\varepsilon}=\lim_{\varepsilon\rightarrow 0}\partial
u^{\varepsilon}(s,X^{\varepsilon}_{s})\sigma^{\varepsilon}(s,X^{\varepsilon})=\partial
u(s,X_{s})\sigma(s,X_{s}), \,\,\,\,\,\ ds\times d\mbox{I\hspace{-.15em}P}%
-a.e.
\end{eqnarray*}
Since for $\P-a.e.$\,\, $\omega,$\,\,\,\
$\partial_{x} u(.,.)$ and $X$ are both continuous, the above equalities
actually holds for all $s\in [0,t]$, \P-a.s.,
proving $(iii)$ and end the proof.
\end{proof}

The following corollary is the direct consequence of the
Theorem $\ref{T1}$. The convention on the generic constant $C>0$ still true.

\begin{corollary}
\label{col}
Assume that the same conditions as in Theorem $\ref{T1}$ hold, and let $%
(X^{t,x},Y^{t,x},Z^{t,x})$ be the solution of FBDSDE $(\ref{a1bis})$-$(\ref{a4})$. Then, there exists a constant $%
C>0$ depending only on $K,\;T,$ and for any $p\geq 1,$ a positive $L^{p}(\Omega,(\mathcal{F}^{t}_s)_{0\leq s\leq t},\P)$-process $\Gamma^{t,x}$, such that
\begin{eqnarray}
|\partial_{x}u(t,x)|\leq C\Gamma_{t}^{t,x},\,\,\,\,\ \forall\ (t,x)\in[0,T]\times%
\mbox{I\hspace{-.15em}R}^{d},\,\,\ \mbox{I\hspace{-.15em}P}-a.s.  \label{cc1}
\end{eqnarray}
Consequently, one has
\begin{eqnarray}
|Z_{s}^{t,x}|\leq C\Gamma_{s}^{t,x}(1+|X_{s}^{t,x}|),\,\,\,\,\,\,\, \forall s\in [0,t],\,\,\,\ %
\mbox{I\hspace{-.15em}P}-a.s.  \label{cc2}
\end{eqnarray}
Furthermore, $\forall\ p>1$, there exists a constant $C_{p}>0$, depending on $%
K,\ T$, and $p$ such that
\begin{eqnarray}
\mbox{I\hspace{-.15em}E}\left\{|X^{t,x}|^{*,p}_{0,t}+|Y^{t,x}|^{*,p}_{0,t}+|Z^{t,x}|^{*,p}_{0,t}\right\}%
\leq C_{p}(1+|x|^{p}).  \label{cc3}
\end{eqnarray}
\end{corollary}

\begin{proof}
We assume first that $p\geq 2$. The case $1<p<2$ then follows
easily from H\"{o}lder inequality. By Lemma $2.1$ and Lemma $2.2$, we can find constant $C>0$ such that
\begin{eqnarray*}
\E\left\{|\nabla X^{t,x}|^{*,p}_{0,t}+|\nabla
Y^{t,x}|^{*,p}_{0,t}+\left(\int_{0}^{T}|\nabla Z_{r}^{t,x}|^{2}dr\right)^{p/2}\right\}\leq C.
\end{eqnarray*}
Then, from the identity $(\ref{c1})$, we deduce immediately that $|\partial_{x}u(t,x)|\leq C\Gamma_{t}^{t,x}$, for all $(t,x)\in[0,T]\times\R$,
where
$$
\Gamma_{s}^{t,x}=\E\left(|\nabla X_{0}^{t,x}|+\int_{0}^{s}[|\nabla X_{r}^{t,x}|+|\nabla Y_{r}^{t,x}|+|\nabla Z_{r}^{t,x}|]dr+\left|\int_{0}^{s}[\nabla X_{r}^{t,x}+\nabla Y_{r}^{t,x}]dB_r\right|\mid\mathcal{F}_{s}^{t}\right).$$ Moreover we get for $s\in[0,t],\;\E(|\Gamma^{t,x}_s|^{p})\leq C$. Then Theorem $\ref{T1}$ $(iii)$ implies that
\begin{eqnarray*}
|Z_{s}^{t,x}|\leq C\Gamma_{s}^{t,x}(1+|X_{s}^{t,x}|),\,\,\,\,\,\ \forall s\in[0,t],\,\,\,\,\ \P\mbox{-a.s.}
\end{eqnarray*}
Now, applying again Lemma $2.1$ and $2.2$ and recalling $(\ref{cc2})$ we get $(\ref{cc3})$, for $p\geq 2$.
\end{proof}

To conclude this section, we would like to point out that in Theorem $3.1$, the functions $f$ and $l$ are assumed to be continuously differentiable in all spatial variables with uniformly bounded partial derivatives, which is much stronger than standing assumption $({\bf A2})$. The following theorem reduces the requirement on $f$ and $l$ to only uniformly Lipschitz continuous, which will be important in our future discussion.
\begin{theorem}
Assume $(\mathbf{A1})$-$(\mathbf{A4})$, and let $(X,Y,Z)$ be the
solution to the FBDSDE $(\ref{a1bis})$-$(\ref{a4})$. Then for all $p>0$, there
exists a constant $C_{p}>0 $ such that
\begin{eqnarray}
\mbox{I\hspace{-.15em}E}\left\{|X|^{*,p}_{0,t}+|Y|^{*,p}_{0,t}+ess\sup_{0
\leq s\leq t}|Z_{s}|^{p}\right\}\leq C_{p}(1+|x|^{p}).  \label{ct1}
\end{eqnarray}
\end{theorem}
\begin{proof}
In the light of the corollary $\ref{col}$, we need only consider $p\geq
2$. By
Lemma $2.1$ and Lemma $2.2$ it follows that for any $p>0$ there exists $%
C_{p} >0$ such that
\begin{eqnarray}
\mbox{I\hspace{-.15em}E}\{|X|^{*,p}_{0,t}+|Y|^{*,p}_{0,t}\}\leq
C_{p}(1+|x|^{p}).  \label{ct2}
\end{eqnarray}
Next, by similar argument of Theorem $3.1$\,$(iii)$, we consider two sequences of smooth functions $\{f^{\varepsilon}\}_{%
\varepsilon}$ and $\{l^{\varepsilon}\}_{\varepsilon}$ with their first order derivatives in $(x,y,z)$ uniformly bounded in $t$ and $\varepsilon$ such that
\begin{eqnarray*}
\lim_{\varepsilon\rightarrow
0}\left\{\sup_{(t,x,y,z)}|f^{\varepsilon}(t,x,y,z)-f(t,x,y,z)|+ \sup_{x}|l^{\varepsilon}(x)-l(x)|\right\}=0.
\end{eqnarray*}
Denoting $(X^{\varepsilon},Y^{\varepsilon},Z^{\varepsilon})$
the unique solution of the corresponding FBDSDEs and applying Corollary $\ref{col}$, we can find for any $p\geq 2$ a constant $C_{p}>0$, independent of $\varepsilon$%
, such that
\begin{eqnarray}
\mbox{I\hspace{-.15em}E}\left(|Z^{\varepsilon}|^{*,p}_{0,t}\right)\leq C_{p}(1+|x|^{p}).
\label{ct3}
\end{eqnarray}
Furthermore, by $(\ref{pc16})$ we know that $\displaystyle{\E\int_{0}^{t}|Z_{s}^{\varepsilon}-Z_{s}|^{2}ds\rightarrow 0}$ as $\varepsilon\rightarrow 0$. Thus, possibly along a sequence say $(\varepsilon_{n})_{n\geq 1}$ we have $\lim_{n\rightarrow \infty}Z^{\varepsilon_{n}}=Z\,\, ds\times d\P$-a.s. Applying Fatou's lemma and recalling $(\ref{ct3})$ we the obtain
\begin{eqnarray*}
\E\left\{ess \sup_{0\leq s\leq t}|Z_{s}|^{p}\right\}\leq C_{p}(1+|x|^{p})
\end{eqnarray*}
which leads to $(\ref{ct1})$, as desired.
\end{proof}

\section{ Representation theorem}
\setcounter{theorem}{0} \setcounter{equation}{0}
In this section we shall prove the first main theorem of the paper. This
theorem can be regarded as an extension of the nonlinear
Feynman-Kac formula obtained by Pardoux-Peng $\cite{PP4}$. It gives a
probabilistic representation of the gradient (rather than the solution itself) of the stochastic viscosity
solution, whenever it exists, to a quasi-linear parabolic stochastic
PDE. Unlike the cases studied in $(\ref{c1})$, in this section, our representation does not depend on the partial derivatives of the functions
$f,l $ and $g$. In this context such representation is the best tool for us to study the path regularity of the process $Z$ in
the BDSDE with non-smooth coefficients. For notational simplicity, we shall drop the superscript $^{t,x}$ from the solution $(X, Y, Z)$ of FBDSDE $(\ref{a1bis})$-$(\ref{a4})$.

To begin with, let us introduce the two important stochastic integrals that will play a key role in the representation:
\begin{eqnarray*}
M^{s}_{r}=\int_{r}^{s}[\sigma^{-1}(\tau,X_{\tau})\nabla
X_{\tau}]^{T}\downarrow dW_{\tau} \label{d1}
\end{eqnarray*}
and
\begin{eqnarray*}
N^{s}_{r}=\frac{1}{s-r}(M_{r}^{s})^{T}[\nabla X_{r}]^{-1},\,\,\,\ 0\leq r<s \leq t.  \label{d6}
\end{eqnarray*}
Let us recall that
\begin{eqnarray}
\E|M_{r}^{s}|^{2p}&\leq& C_{p}
\E\left(\int_{r}^{s}|\sigma^{-1}(\tau,X_{\tau})%
\nabla X_{\tau}|^{2}d\tau\right)^{p} \\
& \leq & C_{p}(s-r)^{p}\E\left(|\nabla
X_{\tau}|_{s,r}^{*,2p}\right)\leq C_{p}(s-r)^{p},  \notag
\label{d2}
\end{eqnarray}
where $C_{p}>0$ is a generic constant.\newline
An other hand, let us define the filtration $\displaystyle{{\bf G}^{t}=\left\{\mathcal{F}_{s}^{B}\otimes\mathcal{F}_{t}^{W},\,\,  0\leq s\leq t\right\}}$ which will play a important role in the proof of the continuity of the process $Z$ in the BDSDE.
\begin{lemma}
For any fixed $t\in [0,T]$ and any $H\in L^{\infty}(\mathbf{F}^{t},[0,T];\R)$ we have
\begin{description}
\item  $(i)$ $
\begin{array}{l}
\E|\int_{0}^{s}\frac{1}{s-r}H_{r}M_{s}^{r}dB_r|<+\infty
\end{array}
$

\item  $(ii)$ for $\P.a.e.,\,\omega \in \Omega $, the
mapping $
\begin{array}{l}
s\mapsto \int_{0}^{s}\frac{1}{s-r}H_{r}(\omega )M_{s}^{r}(\omega
)dB_r(\omega )
\end{array}
$
is continuous on $[0,t]$

\item  $(iii)$ for $\mbox{I\hspace{-.15em}P}.a.e.\ ,\omega \in \Omega $, the
mapping $
\begin{array}{l}
s\mapsto \mbox{I\hspace{-.15em}E}\{\int_{0}^{s}\frac{1}{s-r}%
H_{r}M_{s}^{r}dB_r/\mathcal{G}_{s}^{t}\}(\omega )
\end{array}
$ is continuous on $[0,t]$
\end{description}
\end{lemma}
\begin{proof}
First, for any $0\leq\tau < s\leq t $ we denote
\begin{eqnarray}
A^{s}_{\tau}=\left\{
\begin{array}{l}
\int_{\tau}^{s}\frac{1}{s-r}H_{r}M_{r}^{s}dr,\,\,\,\,\,\,\,\,\ 0\leq\tau<s \\
\\
0,\,\,\,\,\,\,\,\,\,\,\,\,\,\,\,\,\,\,\,\,\,\,\,\,\,\,\,\,\,\,\,\,\,\,\,\,\,%
\,\,\ \,\,\ if\,\ s=\tau.
\end{array}
\right.
\label{v1}
\end{eqnarray}
To simplify notation, when $\tau=0$ we denote $A^{s}_{0}=A_{s}$.
Further, let $\beta$ be such that $\alpha=1-2\beta<\frac{1}{2}$ and $\beta <1$. Consider the random
variable
\begin{eqnarray}
M^{\ast }=\sup_{0\leq t_{1}<t_{2}\leq t}\frac{|M_{t_{1}}^{t_{2}}|}{%
(t_{2}-t_{1})^{\alpha }};
\label{v0}
\end{eqnarray}
then by $(\ref{d2})$ and Theorem $2.1$ of Revuz-Yor $\cite{RZ}$, we see that $%
\displaystyle{\mbox{I\hspace{-.15em}E}[M^{\ast }]^{2}<+\infty }$.

To prove (i) we note that for any $0\leq \tau\leq s\leq t$ by Burkhölder-Gundy- Davis's inequality one has
\begin{eqnarray}
\E|A^{s}_{\tau}|&\leq & C\E\left(\int_{\tau}^{s}\left|\frac{H_{r}M^{s}_{r}}{s-r}\right|^{2}dr\right)^{1/2}\nonumber\\
&\leq& C\E\left(\int_{\tau}^{s}\frac{|H_{r}|^{2}}{(s-r)^{2\beta}}. \frac{|M^{s}_{r}|^{2}}{(s-r)^{2\alpha}}dr\right)^{1/2}\nonumber\\
&\leq& C\E\left(\int_{\tau}^{s}\left|\frac{|H_{r}}{(s-r)^{\beta}}\right|^{2}dr\right)^{1/2}M^{*}\nonumber\\
&\leq&C\E\left(\int_{\tau}^{s}\frac{dr}{(s-r)^{2\beta}}\right)^{1/2}\|H\|_{\infty}M^{*}=C(s-\tau)^{(1/2)-\beta}\E(\|H\|_{\infty}M^{*}),\label{01}
\end{eqnarray}
where $\|.\|$ denotes the norm of $L^{\infty}([0,T])$. Again letting $C>0$ be a generic constant depending only on $\beta$ and $T$, we have
\begin{eqnarray}
\E|A^{s}_{\tau}|&\leq& C\{\E\|H\|_{\infty}^{2}\}^{1/2}\{\E(M^{*})^{2}\}^{1/2}\nonumber\\
&\leq &C\|H\|_{L^{\infty}([0,T]\times\Omega)}\|M^{*}\|_{L^{2}(\Omega)}<\infty.\label{i}
\end{eqnarray}
Setting $\tau=0$ in $(\ref{i})$ we proved $(i)$.

To prove $(ii)$ let $\tau=0$ and observe that, in view of $(i),\, A_s$ is a stochastic integral for $0<s\leq t$. Consequently, the mapping $s\mapsto A_s$ is continuous on [0,t].
It remain to prove $(iii)$. In this fact, we remark that the right-hand side of the inequality $(\ref{01})$ (with $\tau=0$) is clearly in $L^{1}$; thus we check easily that the process $A$ is uniformly integrable. Therefore, by similar step in Ma and Zhang \cite{MZ} (see proof for $(iii)$ of Theorem 4.1) it follows that the ${\bf G}^{t}$-optional projection of $A$, denoting $^{o}A_s=\E(A_s|\mathcal{G}^{t}_{s}),\, s\in[0,t]$, has continuous path. This prove $(iii)$, whence the lemma.
\end{proof}
\begin{theorem}
Assume that the assumptions $(\mathbf{A1})$-$(\mathbf{A4})$ hold, and let $%
(X,Y,Z)$ be the adapted solution to FBDSDE $(\ref{a4})$-$(\ref{a1bis})$. Then
\begin{description}
\item  $\left( i\right) $ the following identity holds $\mbox{I%
\hspace{-.15em}P}$-almost surely:
\begin{eqnarray}
Z_{s} &=&\mbox{I\hspace{-.15em}E}\left\{
l(X_{0})N_{0}^{s}+\int_{0}^{s}f(r,X_{r},Y_{r},Z_{r})N_{r}^{s}dr
+\int_{0}^{s}g(r,X_{r},Y_{r})N_{r}^{s}dB_r|{\mathcal{F}}^{t}_{s}\right\}\sigma(s,X_{s})\nonumber\\
\forall\, 0\leq s\leq t;\label{d7}
\end{eqnarray}
\item  $\left( ii\right) $There exists a version of $Z$ such that for $%
\mbox{I\hspace{-.15em}P}$-a.e.\,\,$\omega \in \Omega $,\ the mapping $s\mapsto
Z_{s}(\omega )$ is continuous;
\item  $\left( iii\right) $ If in addition the functions $f$ and $l$ satisfy
assumptions of Theorem 3.1, then for all $(t,x)\in \lbrack 0,T]\times
\R^{d}$ it holds that
\begin{eqnarray}
\partial _{x}u(t,x) &=&\mbox{I\hspace{-.15em}E}\left\{
l(X_{0})N_{0}^{t}+\int_{0}^{t}f(s,X_{r},Y_{r},Z_{r})N_{r}^{t}dr+\int_{0}^{t}g(r,X_{r},Y_{r})N_{r}^{s}dB_r|\mathcal{F}
_{t}^{B}\right\}.\nonumber\\   \label{d8}
\end{eqnarray}
\end{description}
\end{theorem}
\begin{proof}
Again we shall consider only the case $d=1$. We assume first that $l\in
C_{b}^{1}(\mbox{I\hspace{-.15em}R})$ and $f\in C_{b}^{0,1}([0,T]\times
\mbox{I\hspace{-.15em}R}^{3})$.
Using the nonlinear Feynman-Kac formula of Pardoux and Peng $\cite{PP4}$ we obtain that
for $0\leq s\leq t$,
\begin{eqnarray}
u(s,X_{s})=Y_{s} &=&\mbox{I\hspace{-.15em}E}\left\{
l(X_{0})+\int_{0}^{s}f(r,X_{r},Y_{r},Z_{r})dr+\int_{0}^{s}g(r,X_{r},Y_{r})dB_r|\mathcal{F}^{t}_{s}\right\}.   \label{9}
\end{eqnarray}
Similar arguments to those used in the work of Ma and Zhang $\cite{MZ}$, provide the following:
\begin{eqnarray*}
\partial_{x}u(s,X_{s})=\E\left\{l(X_{0})N_{0}^{s}
+\int_{0}^{s}f(r,X_r,Y_r,Z_r)N^{s}_{r}dr+\int_{0}^{s}g(r,X_r,Y_r)N^{s}_{r}dB_r|
{\mathcal{F}}^{t}_{s}\right\}. \label{q7}
\end{eqnarray*}
In particular, setting $s=t$ we obtain $(\ref{d8})$, this proves
$(iii)$.

We now consider the general case. First we fix
$s\in [0,t]$. For $\varphi = l, f$, let $\varphi^{\varepsilon}
\in C^{\infty}, \varepsilon > 0$, be the mollifiers of $\varphi$,
and let $(Y^{\varepsilon},Z^{\varepsilon})$ be the solution of the
BDSDE in $(\ref{a1bis})$ with coefficients
$(l^{\varepsilon},f^{\varepsilon},g)$. Then for each
$\varepsilon > 0$, as the previous we get
\begin{eqnarray}
Z^{\varepsilon}_{s} &=&\E\left\{
l^{\varepsilon}(X_{0})N_{0}^{s}+\int_{0}^{s}f^{\varepsilon}(r,X,Y^{\varepsilon}_r,Z^{\varepsilon}_r
)N_{r}^{s}dr+\int_{0}^{s}g(r,X,Y^{\varepsilon}_r)N_{r}^{s}dB_r|{\mathcal{F}}^{t}_{s}\right\}\sigma(s,X_{s}).\nonumber\\
\label{q8}
\end{eqnarray}
Passing to limit as $\varepsilon $ goes to zero in $(\ref{q8})$,
we get $(\ref{d7})$ $\P$-a.s., for each fixed $s\in [0,t]$.

We should note that to prove part $(i)$ we still need to show that $(\ref{d7})$ actually holds for all $s\in[0, T]$, $\P$-a.s., but it is easy to see that this will follow from part $(ii)$; that is, the process $Z$ has a continuous version. Thus it remain to prove only $(ii)$.
To do this we first note that
\begin{eqnarray}
Z_s &=&\E\left\{
l(X_{0})N_{0}^{s}+\int_{0}^{s}f(r,X_r,Y_r,Z_r
)N_{r}^{s}dr+\int_{0}^{s}g(r,X_r,Y_r)N_{r}^{s}dB_r|{\mathcal{G}}^{t}_{s}\right\}\sigma(s,X_{s}).\nonumber\\
\end{eqnarray}
Lemma 4.1 in \cite{MZ} and Lemma 4.1 imply that the mapping
\begin{eqnarray*}
s\mapsto \E\left\{\int_{s}^{T}f(r,X_r,Y_r,Z_r
)N_{r}^{s}dr+\int_{0}^{s}g(r,X_r,Y_r)N_{r}^{s}dB_r|{\mathcal{G}}^{t}_{s}\right\}
\end{eqnarray*}
is a.s. continuous on $[0,t]$. By the similar ideas used in Ma and Zhang \cite{MZ} replacing ${\bf F}^{t}$ which here is not a filtration by ${\bf G}^{t}$, it follows that the mapping $\displaystyle{s\mapsto\E\left\{l(X_{0})N_{0}^{s}|{\mathcal{G}}^{t}_{s}\right\}}$
is also continuous on $[0,t]$. Consequently, the right side of $(\ref{d7})$ is
a.s. continuous on $[0,t]$, and hence $(\ref{d7})$ holds for all
$s\in [0,t]$, $\P$-a.s., proving $(ii)$, whence the theorem.
\end{proof}
\begin{remark}
A direct consequence of Theorem 4.2 that might be useful in
application is the following improvement of Theorem 3.3: assume that
$({\bf A1})$ and  $({\bf A2})$ hold, then for all $p > 0,$ there
exists a constant $C_{p}>0$ depending only on $T, K$ and $p$ such
that
\begin{eqnarray}
\E\left\{|X|^{*,p}_{0,t}+|Y|^{*,p}_{0,t}+|Z|^{*,p}_{0,t}\right\}\leq
C_{p}(1+|x|^{p}) \label{q11}
\end{eqnarray}
Indeed, since by Theorem 4.1, $Z$ has a continuous version, thus
$(\ref{ct1})$ becomes $(\ref{q11})$
\end{remark}

\section {Discrete function case}
\setcounter{theorem}{0} \setcounter{equation}{0}
Let us recall that we have proved in Theorem $4.2$ that the process
$Z$ in the solution to the FBDSDE $(\ref{a4})$-$(\ref{a1bis})$ has continuous paths,
under the condition that the coefficients $f$ and $l$ are only
uniformly Lipschitz continuous. While such a result is already an
improvement of that of Pardoux and Peng \cite{PP4}, it still within
the paradigm of the standard FBDSDE in the literature, to wit, the
terminal condition of the BDSDE is of the form $l(X_{0})$ (see also
\cite{PP4}). In this section we consider the class of
BDSDEs whose terminal conditions are path dependent. More precisely,
we assume that the terminal condition of the BDSDE is the form
$\displaystyle{\xi=l(X_{t_{0}},X_{t_{1}},....,X_{t_{n}})}$, where
$0= t_{0}<t_{1}<....<t_{n}=t$ is any partition of $[0,t]$. We shall
prove a new representation theorem for the process $Z$, and will
extend the path regularity result to such a case.
\begin{theorem}
Assume that $({\bf A1})$-$({\bf A3})$ hold; and in ({\bf A3}), $l:R^{d(n+1)}\rightarrow\R$. Let $\pi:0=t_{0}<t_{1}<.....<t_{n}=t$ be
a given partition of $[0,t]$, and let $(X,Y,Z)$ be the unique
adapted solution to the following FBDSDE:
\begin{eqnarray}
X_{s}&=&x+\int^{t}_{s}b(r,X_{r})dr+\int^{t}_{s}\sigma(r,X_{r})dW_{r},\nonumber\\
Y_{s}&=&l(X_{t_{0}},X_{t_{1}},...,X_{t_{n}})+\int^{s}_{0}f(r,X_{r},Y_{r},Z_{r})dr
\label{c2}\\
&&+\int^{s}_{0}g(r,X_{r},Y_{r})dB_{r}-\int^{s}_{0}Z_{r}dW_{r},\,\,\, s\in[0,t].\nonumber
\end{eqnarray}
Then on each interval $(t_{i-1},t_{i}), i=1,....,n$, the
following identity holds:
\begin{eqnarray}
Z_{s}&=&\E\left\{l(X_{t_{0}},X_{t_{1}},...,X_{t_{n}})N^{s}_{t_{i-1}}+\int^{s}_{0}f(r,X_{r},Y_{r},Z_{r})N^{s}_{r\vee t_{i-1}}dr\nonumber\right.\\
&&\left.+\int^{s}_{0}g(r,X_{r},Y_{r})N^{s}_{r\vee t_{i-1}}dB_r|{\mathcal{F}}_{s}^{t}\right\}\sigma(s,X_{s}).
\;\;\;\; s\in(t_{i-1},t_{i}) \label{c3}
\end{eqnarray}
Further, there exists a version of process $Z$ that enjoys the
following properties:
\begin{description}
\item $(i)$ the mapping $s\mapsto Z_{s}$ is $a.s.$ continuous on
each  interval $(t_{i-1},t_{i}),\,\  i=1,.....,n$;
\item $(ii)$ limits $Z_{t_{i}^{-}}=\lim_{s\uparrow t_{i}}Z_{s}$
and $Z_{t_{i}^{+}}=\lim_{s\downarrow t_{i}}Z_{s}$ exist;
\item $(iii)$ $\forall p>0$, there exists a constant $C_{p}>0$
depending only on $T, K$ and $p$ such that
\begin{eqnarray}
\E|\Delta Z_{t_{i}}|^{p}\leq C_{p}(1+|x|^{p})\leq\infty. \label{c4}
\end{eqnarray}
\end{description}
Consequently, the process $Z$ has both càdlàg and càglàd version
with discontinuities $t_{0},...,t_{n}$ and jump sizes satisfying
$(\ref{c4})$
\end{theorem}
\begin{proof}
As before we will consider only the case $d = 1$, and we assume first that $f,\,l\in C^{1}_{b}$. Let us first establish the identity $(\ref{c3})$. We fix an arbitrary index $i$ and consider the interval $(t_{i-1},t_i)$. By virtue of the Malliavin operator $D$, Theorem 2.4 and the uniqueness of the adapted solution to BDSDE, we obtain
\begin{eqnarray}
Z_{s}&=&\sum_{j\geq
i}\partial_{j}lD_{s}X_{t_{j}}+\int^{s}_{0}[f_{x}(r)D_{s}X_{r}+f_{y}(r)D_{s}Y_{r}
+f_{z}(r)D_{s}Z_{r}]dr\nonumber\\
&&+\int^{s}_{0}[g_{x}(r)D_{s}X_{r}+g_{y}(r)D_{s}Y_{r}
+g_{z}(r)D_{s}Z_{r}]dB_{r}-\int^{s}_{0}D_{s}Z_{r}dW_{r}\nonumber\\
&=&\left\{\sum_{j\geq
i}\partial_{j}l\nabla X_{t_{j}}+\int^{s}_{0}[f_{x}(r)\nabla_{s}X_{r}
+f_{y}(r)\nabla^{i}Y_{r}
+f_{z}(r)\nabla^{i} Z_{r}]dr\right.\nonumber\\
&&\left.+\int^{s}_{0}[g_{x}(r)\nabla X_{r}+g_{y}(r)\nabla^{i}
Y_{r} +g_{z}(r)\nabla^{i} Z_{r}]dB_{r}-\int^{s}_{0}\nabla^{i}
Z_{r}dW_{r}\right\}(\nabla
X_{s})^{-1}\sigma(s,X_{s})\nonumber\\
&=&\nabla^{i}Y_{s}(\nabla X_{s})^{-1}\sigma(s,X_{s}),\,\,\,\,\,\,\
t_{i-1}<s<t_{i}.\label{c9}
\end{eqnarray}
Taking the conditional expectation $\E\{.|\mathcal{F}^{t}_{s}\}$ on
two sides of $(\ref{c9})$ we obtain
\begin{eqnarray}
Z_{s}&=&\E\left\{\sum_{j\geq
i}\partial_{j}l\nabla X_{t_{j}}+\int^{s}_{0}[f_{x}(r)\nabla_{s}X_{r}
+f_{y}(r)\nabla_{s}^{i}Y_{r}
+f_{z}(r)\nabla^{i} Z_{r}]dr|{\mathcal{F}}_{s}^{t}\right\}(\nabla
X_{s})^{-1}\sigma(s,X_{s}).\nonumber\\\label{c11}
\end{eqnarray}
The rest of the proof is similar to the BSDE case.
It is clear now that to prove the theorem we need only prove properties $(i)$-$(iii)$, which we will do. Note that $(i)$ is obvious, in light of Theorem 4.2 and thanks to representation $(\ref{c3})$. Property $(ii)$ is a slight variation of Lemma 4.1 and Lemma 4.1 of Ma and Zhang \cite{MZ}, with $0$ there being replaced by $t_{i-1}$, for each $i$. Therefore we shall only check $(iii)$.
To this end, we define $\displaystyle{\Delta Z_{t_{i}}=Z_{t_{i}+}-Z_{t_{i}-}}$. From $(\ref{c9})$ it not difficult to check that
\begin{eqnarray*}
Z_{t_{i}-}=\nabla^{i}Y_{t_{i}}[\nabla X_{t_{i}}]^{-1}\sigma(t_{i},X_{t_{i}})\,\,\,\,
Z_{t_{i}+}=\nabla^{i+1}Y_{t_{i}}[\nabla X_{t_{i}}]^{-1}\sigma(t_{i},X_{t_{i}}).
\end{eqnarray*}
Denoting $\displaystyle{\alpha^{i}_{s}=-(\nabla^{i+1}Y_{s}-\nabla^{i}Y_{s}),\,\, i=1,....,n,}$ we have
\begin{eqnarray}
\Delta Z_{t_{i}}=(\nabla^{i+1}Y_{s}-\nabla^{i}Y_{s})\sigma(t_{i},X_{t_{i}})
=-\alpha^{i}_{t_{i}}\sigma(t_{i},X_{t_{i}}).\label{L1}
\end{eqnarray}
Further, since $(\nabla^{i}Y,\nabla^{i}Z)$ denotes the adapted solution of
the following BDSDE
\begin{eqnarray*}
\nabla^{i}Y_{\tau}&=&\sum_{j\geq
i}\partial_{j}l\nabla X_{t_{j}}+\int^{\tau}_{0}[f_{x}(r)\nabla
X_{r} +f_{y}(r)\nabla^{i}Y_{r}
+f_{z}(r)\nabla Z_{r}]dr\nonumber\\
&&+\int^{\tau}_{0}[g_{x}(r)\nabla X_{r}+g_{y}(r)\nabla^{i} Y_{r}
+g_{z}(r)\nabla^{i} Z_{r}]dB_{r}-\int^{\tau}_{0}\nabla^{i}
Z_{r}dW_{r},\,\,\,\tau\in[t_{i-1},t],
\end{eqnarray*}
if we denote $\displaystyle{\beta{i}_{s}=-(\nabla^{i+1}Z_{s}-\nabla^{i}Z_{s})}$, then we have
\begin{eqnarray}
\alpha^{i}_{s}&=&\partial_{i}l\nabla_{t_{i}}+\int^{s}_{0}[f_{y}(r)\alpha^{i}_{r}+f_{z}(r)\beta^{i}_{r}]dr
+\int^{s}_{0}[g_{y}(r)\alpha^{i}_{r}+g_{z}(r)\beta^{i}_{r}]dB_{r}\nonumber\\&&-\int^{s}_{0}\beta^{i}_{r}dW_{r},\,\,\, s\in[0,t].
\label{L2}
\end{eqnarray}
So $(\alpha^{i},\beta^{i})$ is the adapted solution to the linear BDSDE $(\ref{L2})$. It follows by Lemma 2.2 that $\forall\, p>0$ there exists a $C_{p}>0$ such that $\displaystyle{\E\{|\alpha^{i}_{t_{i}}|^{p}\}\leq C_{p}}$. On the other hand the same estimate holds for $\sigma(s,X_{s})$ because of assumption $({\bf A1})$ and Theorem $3.3$; for $[\nabla X]^{-1}$ since it is solution of a appropriated SDE. It readily seen that $(\ref{c4})$ follows from $(\ref{L1})$ which prove $(iii)$.\newline
Finally, we note that when $f$ and $l$ are only Lipschitz, $(\ref{c3})$ still holds, modulo a standard approximation the same as that in Theorem $4.2$. Thus properties $(i)$ and $(ii)$ are obvious. To prove $(iii)$ we should observe that the standard approximation yield that $\displaystyle{\Delta Z^{\varepsilon}_{t_{i}}}\rightarrow\Delta Z_{t_{i}}$ a.s. So if $(\ref{c4})$ holds for $\Delta Z^{\varepsilon}_{t_{i}}$, then letting $\varepsilon\rightarrow 0$, $(\ref{c4})$ remains true for $\Delta Z_{t_{i}}$, according the Fatou's lemma; that end the proof.
\end{proof}

{\bf Acknowledgments}\newline
We are very grateful to the anonymous referees for their careful reading of the original manuscript and for many useful suggestions


\begin{thebibliography}{99}
\bibitem{B1}  Bismut J.M., Conjuquate convex function in optimal
stochastic control. {\sl J. Math. Anal. App.} {\bf 44}, $384-404$,
$(1973)$.
\bibitem{B2}   Bismut J.M., An introductory approach to duality in
stochastic control. {\sl J. Math. SIAM Rev,} {\bf 20}, $62-78, \
(1978)$.
\bibitem{BM1} Buckdahn R. and Ma J., Stochastic viscosity solution
for nonlinear stochastic partial differential equations, Part I,
stoch. Proc. and their appl. {\bf 93}, $181-204,\
(2001)$.
\bibitem{BM2} Buckdahn R. and Ma J., Stochastic viscosity solution
for nonlinear stochastic partial differential equations, Part II,
stoch. Proc. and their appl. {\bf 93}, $205-228,\
(2001)$.
\bibitem{Cal} Crandall, M.G., Ishii, H. and Lions, P.L., User's guide to viscosity solutions
of second order partialdifferential equations. Bull. Amer. Math. Soc. (NS) {\bf 27}, $1-67$, $1982$.

\bibitem{kal} El Karoui N., Peng S. and Quenez M. C. Backward stochastic
differential equation in finance. {\sl Mathematical finance}. {\bf
7,} $1-71, \ (1997)$.
\bibitem{DE}  Duffie, D. and Epstein, L. Stochastic \ differential utility.
\textsl{Econometrica} \textbf{60}, 353-394, (1992).
\bibitem{HL}  Hamad\`{e}ne S. and Lepeltier J. P., Zero-sum stochastic
differential games and BSDEs. {\sl Systems and Control Letters}.
{\bf 24}, $259-263, \ (1995)$.
\bibitem{KS} Karatzas, I. and Shreve, S.E., Brownian Mation and
Stochastic Calculus, springer, $(1987)$.

\bibitem{LS} Lion P.L. and Souganidis P. E., Fully nonolinear
stochastic partial differential equations, non-smooth equations and
applications, C.R. Acad.Sci.Paris, {\bf 327} serie I, $735-741,\
(1998)$.
\bibitem{MZ} Ma J. and Zhang J., Representation theorems for
Backward stochastic differential Equations, Anal. of Appl. Prob.
{\bf 12: 4}, $1390-1418,\ (2002)$.
\bibitem{PP1}  Pardoux E. and Peng S, Adapted solution of backward
stochastic differential equation.{ \sl Syst. cont. Lett.} {\bf
4}, $55-61,\ (1990)$.
\bibitem{PP2}  Pardoux E. and Peng S. Backward stochastic differential
equation and quasilinear parabolic partial differential equations.
In: {\sl B. L.Rozovski, R. B. Sowers (eds). Stochastic partial
equations and their applications.  Lect. Notes control Inf. Sci.}
{\bf 176}, $200-217$, Springer, Berlin, $(1992)$.
\bibitem{PP3} Peng S., Probabilistic interpretation for system of
quasi-linear parabolic equations, Stochastic {\bf 37}, $61-74,
(1991)$.
\bibitem{PP4} Pardoux E. and Peng S., Backward doubly stochastic differential
equations and systems of quasilinear SPDEs, Probability Theory and
related field ${\bf 98}$, $209-227$, $(1994)$.

\bibitem{RZ} Revus, D. and Yor, M., Brownian Motion and Continuous
Martingales, Springer, $(1991)$.


\end{thebibliography}
\end{document}